\documentclass[10pt, a4paper]{amsart}

\usepackage[T1]{fontenc}
\usepackage{lmodern}

\usepackage[all,cmtip]{xy}
\usepackage{array}
\usepackage{amsmath, amssymb}
\usepackage{graphics}
\usepackage{caption}
\usepackage{subcaption}
\usepackage{algorithm}
\usepackage{algorithmic}
\usepackage{booktabs}
\usepackage{color}
\usepackage{cancel}
\usepackage{caption}
\usepackage{subcaption}
\usepackage{amsthm}

\usepackage[utf8]{inputenc}
\usepackage[normalem]{ulem}

\usepackage[T1]{fontenc}

\usepackage{tikz}

\newtheorem{theorem}{\bf Theorem}
\newtheorem{lemma}[theorem]{\bf Lemma}
\newtheorem{proposition}[theorem]{\bf Proposition}

\newtheorem{remark}[theorem]{\bf Remark}

\newtheorem{example}[theorem]{\bf Example}
\newtheorem{corollary}[theorem]{\bf Corollary}

\newcommand{\cis}[1][C]{\ensuremath{\mathbb{#1}}}
\newcommand{\av}[1]{\ensuremath{\mathcal{#1}}}

\newcommand{\f}[1]{\mathbf{#1}}

\newcommand{\euR}[1]{\ensuremath{\mathbb{E}^{#1}_{\cis[R]}}}

\newcommand{\C}{\mathbb{C}}
\newcommand{\R}{\mathbb{R}}
\newcommand{\N}{\mathbb{N}}

\newcommand{\diff}{\mathrm{d}}

\newcommand{\re}{\mathrm{Re}\,}

\newcommand{\I}{\mathrm{i}}
\newcommand{\e}{\mathrm{e}}
\newcommand{\id}{\mathrm{id}}

\newcommand{\SO}[1]{\ensuremath{\mathbf{SO}({\R,#1}})}
\newcommand{\OO}[1]{\ensuremath{\mathbf{O}({\R,#1}})}

\newcommand{\iso}[1]{\mathbf{Iso}_{#1}}
\newcommand{\sym}{\mathbf{Symm}}
\newcommand{\simm}{\mathbf{Sim}}

\title[Symmetries of plane curves]{Symmetries and similarities of planar algebraic curves using harmonic polynomials}

\author[J. G. Alc\'{a}zar]{Juan Gerardo Alc\'azar }
\author[M. L\'{a}vi\v{c}ka]{Miroslav L\'{a}vi\v{c}ka}
\author[J. Vr\v{s}ek]{Jan Vr\v{s}ek}

\address[J. Alca\'{a}zar]{Universidad de Alcal\'{a}}
\email{juange.alcazar@uah.es}

\address[M. L\'{a}vi\v{c}ka, J. Vr\v{s}ek]{University of West Bohemia, Pilsen}
\email{\{lavicka, vrsekjan\}@kma.zcu.cz}

\begin{document}

\maketitle


\begin{abstract}
We present novel, deterministic, efficient algorithms to compute the symmetries of a planar algebraic curve, implicitly defined, and to check whether or not two given implicit planar algebraic curves are similar, i.e. equal up to a similarity transformation. Both algorithms are based on the fact, well-known in Harmonic Analysis, that the Laplacian operator commutes with orthogonal transformations, and on efficient algorithms to find the symmetries$/$similarities of a harmonic algebraic curve$/$two given harmonic algebraic curves. In fact, we show that in general the problem can be reduced to the harmonic case, except for some special cases, easy to treat.
\end{abstract}

\section{Introduction.} \label{intro}

This paper addresses, first, the problem of deterministically computing the symmetries of a given planar algebraic curve, implicitly defined, and second, the problem of deterministically checking whether or not two implicitly given, planar algebraic curves are similar, i.e. equal up to a similarity transformation. Both problems have been addressed in many papers coming from applied fields like Computer Aided Geometric Design, Pattern Recognition or Computer Vision; the interested reader can check the bibliographies of the papers \cite{A13, ADTH, AHM14b} for an exhaustive list. In fields like Patter Recognition or Computer Vision the problem of detecting similarity is essential because objects must be recognized regardless of their position and scale. In Computer Aided Geometric Design, symmetry is important on its own right, since it is a distinguished feature of the shape of an object. But it is also important in terms of storing or managing images, because knowing the symmetries of an image allows the machine to reconstruct the object at a lower computational or memory cost.

In applied fields, the problems treated in this paper are considered mostly for curves defined in a floating point representation. This makes perfect sense in many applications, since the considered curves are often approximate models for real objects. However, here we assume to be dealing with exact curves, implicitly defined by polynomials with coefficients in some real computable field. Recent, or relatively recent, publications also treating the exact case, without considering floating point coefficients, are \cite{ADTH, AHM14, AHM14b, HJ17, L09, LR08}. The papers \cite{AHM14,LR08} address symmetries of planar algebraic curves; in \cite{AHM14} the curve is assumed to be defined by a rational parametrization, while in \cite{LR08} the input is defined by means of an implicit equation. Similarities of planar algebraic curves are considered in \cite{AHM14b}, where the input is a pair of rational parametrizations, \cite{ADTH}, where the input is a pair of implicit equations, as in our case, and \cite{HJ17}, where the input is also a pair of rational parametrizations. Nevertheless, the problem addressed in \cite{HJ17} is more general, since the authors consider how to deterministically recognize whether two rational curves of any dimension are related by some, non-necessarily orthogonal, projective or affine transformation.

The algorithms presented in this paper to solve both the symmetry and the similarity problem lie on the fact that the Laplacian operator commutes with orthogonal transformations.
Therefore, symmetry$/$similarity are preserved when we take laplacians. Furthermore, the laplacian of a polynomial is another polynomial whose degree is lowered by two with respect to the degree of the original polynomial. Thus, by repeteadly taking laplacians we come down to either harmonic polynomials, i.e. polynomials whose laplacian is identically zero, or conics, or lines. When we end up with harmonic polynomials we will see that the problem boils down to computing with univariate, complex polynomials. When we end up with conics or lines the problem, except in some special cases, can be reduced to the harmonic case. This means that we can replace our input by a new input consisting only of harmonic polynomials, therefore reducing the computation to the case before. In some special cases this is not enough; however, in those cases we end up either with a conic or with the product of a line and a cirle, so the problem can be solved by elementary, Linear Algebra methods. 

The similarity problem is certainly more general than the symmetry problem, since determining the symmetries of an algebraic curve can be reduced to finding the self-similarities of the curve (see Proposition 2 in \cite{AHM14b}). However, we believe that for sake of clarity it is preferable to address first the problem for symmetries, and move afterwards to the case of similarities. In both cases we start analyzing in detail the problem for harmonic polynomials. Then we show how to reduce the general case to the harmonic case. Additionally, we develop some results on symmetries of harmonic polynomials that can be interesting on its own; in particular, we provide closed forms for the tentative axes of symmetry or the tentative center of symmetry (except in some of the special cases mentioned above) of an algebraic curve defined by a harmonic polynomial.  

Let us also compare our results with the results in the papers mentioned in the preceding paragraphs of this Introduction; we restrict to the papers where implicitly defined curves are considered. Regarding the detection of rotational symmetry, our results are not necessarily more efficient than those in \cite{LR08}, although our method is, in our opinion, clearer and easier to implement. The same comment can be made on detecting reflectional symmetry, not treated in \cite{LR08}, but addressed in the Ph. D. Thesis \cite{L09} of one of the authors. Regarding similarity detection, in \cite{ADTH} the problem is reduced to solving a bivariate polynomial system with the extra advantage of containing a univariate equation in the generic case, which makes the computation efficient. In fact, \cite{ADTH} contains an extensive analysis of examples and timings, showing an overall good performance even for serious inputs. Our algorithm has advantages and disadvantages, compared to the algorithm in \cite{ADTH}. The main advantage is computational efficiency: in our case, we completely avoid bivariate polynomial solving, and we can manage serious inputs (see Remark \ref{efficiency-5} and Remark \ref{rem-efficiency-again} in Section \ref{sec-similarities}) in less than 1 second, which makes our algorithm much faster than the algorithm in \cite{ADTH}. We can mention, however, two disadvantages: first, the algorithm in \cite{ADTH} can be easily adapted to the case of floating-point coefficients, while our algorithm is worse suited for that. Secondly, the algorithm in \cite{ADTH} can be generalized to detect whether or not two given algebraic curves are affinely equivalent, i.e. related by an affine, non-necessarily orthogonal, transformation. This, however, is not so clear for our algorithm, since laplacians are not preserved by non-orthogonal transformations. 

The structure of the paper is the following. Section \ref{sec-generalities} provides some general results, including some background on laplacians, harmonic polynomials, symmetries and similarities, to be used later in the paper. Symmetry computation is addressed in Section \ref{sec-symmetries}, first for harmonic polynomials, then for the general case. The same structure is used in Section \ref{sec-similarities} in order to address similarity detection. Our conclusions and some lines for further research are presented in Section \ref{sec-conclusion}.

\section{Symmetries and similarities of algebraic curves, and Laplace operator.}\label{sec-generalities}

\subsection{Symmetries.} \label{subsec-sym}

An~\emph{isometry} of the Euclidean plane $\euR{2}$ is a~mapping $\phi: \euR{2}\rightarrow\euR{2}$ which preserves distances. It can be always written in the form
\begin{equation}\label{eq isometry}
  X\mapsto A\cdot X+\overrightarrow{b},
\end{equation}
where $A\in\OO{2}$ is an orthogonal matrix and $\overrightarrow{b}\in\R^2$. Hence the group of  isometries, denoted by $\iso{2}$ is isomorphic to the~semidirect product $\OO{2}\rtimes\R^2$. The isometries preserving also the orientation of the plane are called \emph{direct} isometries. They form a subgroup $\iso{2}^+\cong\SO{2}\rtimes\R^2$ of $\iso{2}$. Isometries reversing orientation are called \emph{opposite}; the set of opposite isometries is denoted $\iso{2}^-$. Isometries can be defined in $\euR{n}$ in an analogous way, for $n\geq 3$; we represent isometries in $\euR{n}$ by $\iso{n}$.

A planar algebraic curve $\av{C}$ is the zeroset of a polynomial $f(x,y)$. We will assume that $f$ has real coefficients and is square-free, i.e. that $f$ has no multiple factors. {Additionally, even though commonly our input will be a curve $\av{C}$ with infinitely many real points, i.e. a \emph{real} curve, at several points we will be considering $\av{C}$ as a subset of ${\Bbb C}^2$.} Let $\sym(\av{C})$ denote the set of symmetries of the curve, i.e,
\begin{equation}
  \sym(\av{C}):=\{\phi\in\iso{2}\mid\,\phi(\av{C})=\av{C}\},
\end{equation}
which is a subgroup of $\iso{2}$. Also let $\sym^+(\av X)$  be the subgroup of direct  symmetries of the curve. {Now while the notion of symmetry is of {\it geometric} nature, we want to translate this notion into a notion of {\it algebraic} nature. In order to do this,} we define the set $\sym(f)$ of symmetries of the \emph{polynomial} $f$ defining $\av{C}$ as
\begin{equation}
  \sym(f):=\{\phi\in\iso{2}\mid\,f\circ \phi=\lambda  f\},
\end{equation}
where $\lambda\neq 0$ is a constant. Then we have the following result.

\begin{lemma} \label{lem-recognize}
If $f$ is square-free, then $\sym(\av{C})=\sym(f)$.
\end{lemma}

\begin{proof} Let us see $\sym(\av{C}) \subset \sym(f)$, first. By definition, $\phi\in \sym(\av{C})$ iff ${\bf x}\in \av{C}\Leftrightarrow \phi({\bf x})\in \av{C}$. Therefore, if $\phi\in \sym(\av{C})$ then $f$ and $f\circ \phi$ define the same variety. Since $f$ and $f\circ \phi$ have the same degree, it follows that $f\circ \phi=\lambda  f$, with $\lambda$ a nonzero constant. Now let us see $\sym(f) \subset \sym(\av{C})$. Indeed, let $\phi\in \sym(f)$. Since $f\circ \phi=\lambda f$, we have that $f(\phi({\bf x}))=0$ iff $f({\bf x})=0$, and hence $\phi(\av{C})=\av{C}$.
\end{proof}

\begin{remark}  Lemma~\ref{lem-recognize} {requires ${\mathcal C}$ to be considered as a subset of ${\Bbb C}^2$; in fact, Lemma~\ref{lem-recognize}} is not necessarily true for $\av{C}_{{\Bbb R}}=\av{C}\cap {\Bbb R}^2$. For example, no real point satisfies an~equation $f(x,y):=x^2+y^2+1=0$. Thus the set of symmetries of $\av{C}_{{\Bbb R}}$ in this case form the full group $\iso{2}$ whereas $\sym(f)=\OO{2}$ only. However, if all the real factors of $f$ define real curves, then Lemma \ref{lem-recognize} is also true for $\av{C}_{{\Bbb R}}$. Observe as well that any symmetry of $\av{C}$ is in particular a symmetry of $\av{C}_{{\Bbb R}}$.
\end{remark}

\begin{remark} The hypothesis on $f$ being square-free is necessary. Indeed, if $f$ is not square-free, one can show that $f$ and $f\circ \phi$ have the same irreducible factors, so the zero--set of $f$ and $f\circ \phi$ is certainly the same. But the multiplicities may be exchanged \cite{Goldman-personal-com}. For instance, $x^2y^3=0$ defines an algebraic curve, symmetric with respect to the line $y=x$; however, when we apply the symmetry $\phi(x,y)=(y,x)$ to the curve, we get the curve $y^2x^3=0$. The zerosets of both curves $x^2y^3=0$ and $y^2x^3=0$ certainly coincide, but the defining polynomials are not symmetric, because the multiplicities of the irreducible factors do not match.
\end{remark}

\begin{remark} \label{high-form} Let $\phi(X)=AX+\overrightarrow{b}$ be a symmetry of the polynomial $f$, and let $\psi(X)=AX$. Writing $f=f_N+f_{N-1}+\cdots +f_0$, where $f_i$ denotes the homogeneous form of $f$ of degree $i$ and $N=\mbox{deg}(f)$, the condition $f\circ \phi=\lambda f$ implies $f_N\circ \psi=\lambda f_N$, i.e. $\psi$ is a symmetry of $f_N$.
\end{remark}

As we will see later, we will be interested in symmetries $\phi$ satisfying that $\phi^k=\id$, where $\id$ represents the identity, and $k\in\N$; this is the case, for instance, of reflections or rotations by angle $2\pi/k$. In this situation, one can show that the coefficient $\lambda$ in \eqref{eq:test of symetry} cannot attain arbitrary values.

\begin{lemma} \label{lem:lambda}
Let $\phi$ be a symmetry satisfying that $\phi^k=\id$ for some $k\in\N$. Then either $f\circ\phi=f$ or $f\circ\phi=-f$. Moreover, if $k$ is odd then $f\circ \phi=f$.
\end{lemma}

\begin{proof}
For $k=2$, $f\circ \phi^2=(f\circ\phi)\circ\phi=(\lambda f)\circ\phi=\lambda^2 f$. Using induction on $k$, we arrive at $\lambda^k=1$. Since $\lambda$ is a real number, when $k$ is even we get $\lambda=\pm 1$; when $k$ is odd, necessarily $\lambda=1$.
\end{proof}

\subsection{Similarities.} \label{subsec-simil}

A~\emph{similarity} of the Euclidean plane $\euR{2}$ is a~mapping $\phi: \euR{2}\rightarrow\euR{2}$ which preserves ratios of distances. It can always be written in the form
\begin{equation}\label{eq similarity}
  X\mapsto \mu A\cdot X+\overrightarrow{b},
\end{equation}
where $\mu\in {\Bbb R}-\{0\}$, $A\in\OO{2}$ is an orthogonal matrix and $\overrightarrow{b}\in\R^2$; we call $\mu$ the \emph{scaling factor} of the similarity. Therefore, any similarity is the composition of a scaling (i.e. a homothety) and an isometry. Again planar similarities form a group $\simm_2$, and we can distinguish similarities preserving the orientation, or \emph{direct}, and similarities reversing the orientation, or \emph{opposite}. Additionally, similarities can be defined in $\euR{n}$ in an analogous way, for $n\geq 3$; we represent similarities in $\euR{n}$ by $\simm_n$.

Given two algebraic curves $\av{C}_1,\av{C}_2$ implicitly defined by two square-free polynomials $f_1,f_2$, and considering, as in the previous subsection, $\av{C}_i\subset {\Bbb C}^2$ for $i=1,2$, we define:
\begin{equation}
  \simm(\av{C}_1,\av{C}_2):=\{\phi\in\simm_2\mid\,\phi(\av{C}_1)=\av{C}_2\},
\end{equation}
and
\begin{equation}
  \simm(f_1,f_2):=\{\phi\in\simm_2\mid\,f_1\circ \phi=\lambda  f_2\},
\end{equation}
where $\lambda\neq 0$ is a constant. Then we have the following result, which can be proven as Lemma \ref{lem-recognize}.

\begin{lemma} \label{lem-recognize-2}
If $f_1,f_2$ are square-free, then $\simm(\av{C}_1,\av{C}_2)=\simm(f_1,f_2)$.
\end{lemma}

Direct similarities can be identified with complex transformations $\phi(z)=az+b$, and opposite similarities with complex transformations $\phi(z)=\alpha\overline{z}+\beta$, where $\alpha,\beta\in {\Bbb C}$, $\alpha\neq 0$. Furthermore, if $\av{C}_1=\av{C}_2=\av{C}$, whenever $\av{C}$ is not the union of (possibly complex) concurrent lines, the self-similarities of $\av{C}$ are  isometries (see Prop. 2 in \cite{AHM14b}).

\subsection{Laplacian, harmonic functions, and the general strategy.}\label{subsec-laplacian}

Let $\R_k[x_1,\dots,x_n]$ denote the vector space of polynomials of degree at most $k>0$ in the variables $x_1,\ldots,x_n$; then the \emph{Laplace operator}, or \emph{Laplacian}, is defined as the linear mapping
\begin{equation}
   \triangle:\R_k[x_1,\dots,x_n]\rightarrow \R_{k-2}[x_1,\dots,x_n],\qquad \triangle f=\sum_{i=1}^n\frac{\partial^2\,f}{\partial\,x_i^2}
\end{equation}
The kernel of the Laplacian consists of the \emph{harmonic polynomials}, which are the polynomials satisfying $\triangle f=0$.

The key property of the Laplacian, in our context, is that the Laplacian commutes with orthogonal transformations (see Chapter 1 of \cite{ABR}). Therefore, for $\phi\in\iso{n}$ we have
\begin{equation}\label{eq:laplacian-symmetries}
\triangle(f\circ \phi)=\triangle f \circ \phi.
\end{equation}

\noindent The next lemma shows that the above equation must, in general, be slightly modified for $\phi\in \simm_n$.

\begin{lemma} \label{Laplacian-similarities}
Let $\phi\in \simm_n$, where $\phi$ is an in Eq. \eqref{eq similarity}, with $\mu$ the scaling factor. Then we have
\begin{equation}\label{eq:laplacian-similarities}
\triangle(f\circ \phi)=\mu^2\cdot (\triangle f \circ \phi).
\end{equation}
\end{lemma}

\begin{proof} By using the Chain Rule one can easily check that Eq. \eqref{eq:laplacian-similarities} holds for translations; therefore, from Eq. \eqref{eq:laplacian-symmetries} we deduce that Eq. \eqref{eq:laplacian-similarities} holds for isometries as well. Thus, it is enough to prove the result for similarities of the type $\phi(x,y)=(\mu x, \mu y)$. This follows, again, from the Chain Rule.
\end{proof}

Combining the previous properties with Lemma \ref{lem-recognize} and Lemma \ref{lem-recognize-2}, we have the following result.

\begin{theorem}\label{thm:symmetry preservation} The following statements hold:
\begin{itemize}
\item [(1)] For $f\in {\Bbb R}[x,y]$, $\sym(f)\subset\sym(\triangle f)$.
\item [(2)] For $f_1,f_2\in {\Bbb R}[x,y]$, $\simm(f_1,f_2)\subset \simm(\triangle f_1,\triangle f_2)$.
\end{itemize}
\end{theorem}

\noindent Notice that in general, $\sym(f)$ (resp. $\simm(f_1,f_2)$) will be a proper subgroup of $\sym(\triangle f)$ (resp. $\simm(\triangle f_1,\triangle f_2)$).

Now let us present the general strategy of our approach. We will present it for symmetries; the case of similarities is analogous. From Lemma \ref{lem-recognize}, in order to check whether or not $\f x\mapsto A\cdot\f x+\f b$ is a symmetry of $\av{C}$, one just needs to verify whether or not
\begin{equation}\label{eq:test of symetry}
  f(A\cdot \f x+\f b)=\lambda f(\f x)
\end{equation}
holds for some $\lambda \neq 0$. However, while Eq. \eqref{eq:test of symetry} is useful to recognize whether the hypersurface defined by $f$ possesses or not a given symmetry, using this equation to find all the~possible symmetries would lead to a~practically unsolvable system of nonlinear equations. Hence our plan is the following:

\begin{center}
 {\it Reduce the possible symmetries of a given curve or surface to finitely many candidates $\phi_1,\dots,\phi_k\in\iso{2}$, which can be tested afterwards.}
\end{center}

Symmetries of curves of degree one or two can be found by elementary methods, but for higher degrees we run into difficulties. Fortunately, applying the Laplacian we can reduce the degree. Therefore, our plan is refined into the following one:

\begin{center}
 {\it Apply the Laplacian operator until we reach either a polynomial of degree 1, or a polynomial of degree 2, or a harmonic polynomial.}
\end{center}

We will see that, except in some trivial types of curves to be excluded from our study, in each case we can find finitely many candidates symmetries for $\av{\mathcal C}$, which will be tested afterwards using Eq. \eqref{eq:test of symetry}. Even more, in order to compute the candidates we just need to study either harmonic polynomials, or simple curves whose symmetries can be derived by elementary methods (conic curves, unions of a circle and a line).

The case of similarities is analogous, but with Eq. \eqref{eq:test of symetry} being replaced by the following equation:
\begin{equation}\label{eq:test of similarity}
  f_1(\mu A\cdot \f x+\f b)=\lambda f_2(\f x).
\end{equation}

\section{Computation of the symmetries.} \label{sec-symmetries}

Let ${\mathcal C}$ be an algebraic plane curve, implicitly defined by a square-free polynomial. It is well known (see Corollary 4 in \cite{AHM14}) that the only algebraic curves with infinitely many symmetries are the unions of parallel lines and the unions of concentric circles. Curves $f(x,y)=0$ corresponding to unions of parallel lines can be recognized by checking whether or not there exists a constant vector $\bar{u}$ such that $\nabla f\cdot \bar{u}=0$, where $\nabla f$ represents the gradient vector of $f$; in turn, this amounts to solving a linear system. Also, curves $f(x,y)=0$ corresponding to unions of concentric circles can be recognized by applying similar ideas to those in \cite{AG17} or \cite{Vrseck} to detect surfaces of revolution in 3-space. Therefore, we will exclude from our study curves which are the unions of parallel lines, and curves which are the unions of concentric circles. As a consequence we will assume that ${\mathcal C}$ has a finite group of symmetries. The following lemma is a rephrasing of the classification theorem of the finite groups of symmetries of the Euclidean plane.  

\begin{lemma}\label{lem:curves finite group}
If $\sym^+(\av C)$ is a non-trivial finite group of direct symmetries of $\av{C}$, then it is a~group of symmetries of a~regular $n$-gon for some $2\leq n\leq \deg\av C$.
\end{lemma}

The elements of a finite symmetry group are rotations, that will be denoted as $\rho_{\f p,\varphi}$, where $\f p$ is the center of rotation and $\varphi$ is the rotation angle, and reflections $\sigma_L$, where $L$ is the reflection axis.  We say that $\av{C}$ has rotational (resp. reflectional) symmetry if $\sym(\av C)$ contains at least one rotation (resp. reflection). Rotations are direct isometries, while reflections are opposite.



The answer provided by our method is the set of rotational symmetries and the set of reflectional symmetries of ${\mathcal C}$, or the proof of its non existence. In order to do this, we provide the candidates for the rotational and reflectional symmetries. Then we can use Eq. \eqref{eq:test of symetry}, with $\lambda=\pm 1$, for testing these candidates.

We will see that we can reduce the computation of the symmetries of ${\mathcal C}$ to the case of harmonic polynomials except for some special situations, that can be easily solved. For this reason, we first address the case of harmonic polynomials, and then we show how to reduce the other cases to this first case.

\subsection{Symmetries of harmonic polynomials}\label{subsec-harmonic}

Throughout this subsection we will assume that $\av{C}$ is defined by a non--constant harmonic polynomial $h\in\R[x,y]$, i.e., satisfying $\triangle h=0$, of degree $\mbox{deg}(h)>1$. Moreover, recall that the Laplace operator in polar coordinates $(r,\theta)$ is expressed, for a bivariate function $f=f(r,\theta)$, as
\begin{equation}\label{eq:laplace polar}
  \triangle f=\frac{1}{r}\frac{\partial}{\partial r}\left(r\frac{\partial f}{\partial r}\right)+\frac{1}{r^2}\frac{\partial^2 f}{\partial\theta^2}.
\end{equation}

\begin{lemma}\label{unions} If $h(x,y)$ is harmonic and $\mbox{deg}(h)>1$, then the curve $h(x,y)=0$ is not a union of two or more parallel lines, or a union of concentric circles.
\end{lemma}

\begin{proof}
Assume that $h(x,y)=0$ is a union of parallel lines. Since the laplacian commutes with orthogonal transformations, we can rotate the parallel lines so that the lines are parallel to the~$y$-axis. Thus asume w.l.o.g. that $h(x,y)=\sum_{i=0}^na_ix^i$. Then the Laplacian is nothing but a second derivative with respect to the variable $x$ and we see that the only harmonic polynomials are of the form $h(x,y)=a_1 x+a_0$.

Assume now that $h(x,y)=0$ is a union of concentric circles. Analogously to the previous case, we can assume w.l.o.g. that all the circles are centered at the origin. Using polar coordinates $(r,\theta)$, the polynomial $h$ transforms into $\sum_{i=0}^na_ir^i$. Using \eqref{eq:laplace polar} we see that $\triangle h=0$ if and only if $h$ is constant.
\end{proof}

\begin{corollary}\label{corol-finite} If $h$ is harmonic and $\deg h>1$, then $\sym(h)$ is finite.
\end{corollary}

Instead of studying the symmetries of $\av{C}$ we will focus on the symmetries of the~ {real} singular points of the vector field $\overrightarrow{v}(x,y)=(\partial_x h,-\partial_y h)$, where $\partial_x h,\partial_y h$ represent the partial derivatives of $h$ with respect to the variables $x,y$. We say that the point $\f q\in\R^2$ is a \emph{singular point} of the vector field $\overrightarrow{v}$, if $\overrightarrow{v}(\f q)=(0,0)$. Now for every $\phi\in\sym(h)$ and from Lemma \ref{lem:lambda}, we have $h\circ\phi=\pm h$; additionally, using the Chain Rule, we get $(\nabla h)\circ \phi=\pm J_\phi\cdot \nabla h$, where $\nabla h$ is the gradient of $h$, and $J_\phi$ is the Jacobian of $\phi$. Then $(\nabla h)\circ \phi$ vanishes iff $\nabla h$ vanishes too, so $\phi$ maps {real} singular points of $\overrightarrow{v}$ onto {real} singular points of $\overrightarrow{v}$. As a consequence, $\sym(h)$ is a subgroup of the symmetry group of the {real} singular points of $\overrightarrow{v}$. As we will see, the set of {real} singular points is always finite {and non-empty} for a harmonic polynomial $h$ and thus we can use the following result.

\begin{lemma}\label{lem:center}
  Let $Q=\{\f q_1,\ldots,\f q_N\}\subset\R^2$ be a finite set of points. Then the center of any rotational (all the axes of reflectional) symmetry of $Q$ is (pass through) the barycenter
  \begin{equation}
     \frac{1}{N}\sum_{i=1}^N\f q_i.
  \end{equation}
\end{lemma}

\begin{remark}
In Fluid Dynamics, the vector field $\overrightarrow{v}$ is known as \emph{complex velocity}, and $h$ is called the \emph{velocity potential} (see \cite{Batch}). Different velocity potentials are used to model different flows. Additionally, in that context the singular points of $\overrightarrow{v}$ are called \emph{stagnation points}.
\end{remark}

Let $S_{\overrightarrow{v}}$ denote the set of {real} singular points of the vector field. In order to find the group $\sym(S_{\overrightarrow{v}})$, we first identify $\R^2$ with $\C$ via $(x,y)\leftrightarrow x+\I y$. The vector field
$\overrightarrow{v}(x,y)$ can be naturally replaced by a~complex function
\begin{equation}\label{assoc-g}
  g(x,y)=\partial_x h-\I \partial_y h.
\end{equation}
Observe that since $h(x,y)$ is nonconstant by hypothesis, $g(x,y)$ is not identically zero. The~standard substitution
\begin{equation}\label{eq:subst}
   x=\frac{1}{2}(z+\overline{z})\qquad\text{and}\qquad  y=-\frac{\I}{2}(z-\overline{z})
\end{equation}
allows to write $g(x,y)$ as a complex function $g(z,\overline{z})$ in the complex variable $z$. Using the harmonicity of $h$ one can show that $g(x,y)$ satisfies the Cauchy-Riemann conditions, and therefore that $g(x,y)$ is holomorphic. Thus, $g(z,\overline{z})$ does not depend on $\overline{z}$, i.e. $g(z,\overline{z})=g(z)$, where
\begin{equation}\label{assoc-z}
  g(z)=\sum_{j=0}^{n-1}a_jz^j=a_{n-1}\prod_{j=1}^{n-1}(z-\xi_j).
\end{equation}
Since $g(z)$ is not identically zero, $g(z)$ has finitely many roots. Therefore $\partial_x h,\partial_y h$ cannot have any common factor. {Additionally, $g(z)$ is not constant, so $S_{\overrightarrow{v}}$ is non-empty. Therefore,} we have the following result.

\begin{lemma} \label{on-harmonic}
Let $h(x,y)$ be a nonconstant harmonic bivariate polynomial. Then we cannot write $h(x,y)=\tilde{h}(x,y)+c$, where $\tilde{h}(x,y)$ has multiple factors, and $c\in \R$. In particular, $h(x,y)$ must be square-free, and the set of singular points of $h(x,y)=0$ is finite {and non-empty}.
\end{lemma}

Notice that since $h(x,y)$ is square-free, we are in the hypotheses of Lemma \ref{lem-recognize}. We will say that $g(z)$ is \emph{associated with} $h(x,y)$. The following lemma shows how to obtain $h(x,y)$ from a given complex polynomial $g(z)$.

\begin{lemma}\label{G-z}
Let $g(z)$ be a nonzero polynomial, and let $h(x,y)$ be a harmonic polynomial such that $g(z)$ is associated with $h(x,y)$. Then there exists a primitive $G(z)$ of $g(z)$ such that $h(x,y)=\re(G(z))$.
\end{lemma}

\begin{proof} Let $\tilde{G}(z)$ be a primitive of $g(z)$, and let $\tilde{G}(x+iy)=U(x,y)+\I V(x,y)$. Since $\tilde{G}(z)$ is holomorphic, we have $\tilde{G}'(z)=g(z)=\partial_x U+\I \partial_x V=\partial_x U-\I \partial_y U$. On the other hand, by definition $g(z)=\partial_x h-\I \partial_y h$. Therefore, we have $\partial_x U=\partial_x h$, $\partial_y U=\partial_y h$ and thus $h=U+c$, where $c$ is a real constant. Hence $h(x,y)=\re(\tilde{G}(z)+c)$, where $G(z)=\tilde{G}(z)+c$ is also a primitive of $g(z)$.
\end{proof}

Now we are ready to relate $\sym(h)$ with the polynomial $g(z)$. Since the singular set $S_{\overrightarrow{v}}$ corresponds exactly to the roots of $g(z)$ and the sum of the roots of $g(z)$ is encoded in the coefficient $a_{n-2}$ via the Cardano-Vieta's formulae, Lemma~\ref{lem:center} directly provides the following result.

\begin{theorem}\label{thm:barycenter}
  Let $h(x,y)$ be a harmonic polynomial defining a curve invariant under the rotation $\rho_{\varphi,\f p}$. If $g(z)$ is the polynomial constructed before, and the center $\f p$ is identified with a complex number, then
  \begin{equation}\label{eq:center harmonic}
      \f p =-\frac{a_{n-2}}{(n-1)a_{n-1}}.
  \end{equation}
\end{theorem}

\begin{remark}\label{tentative-angle}
By Lemma~\ref{lem:curves finite group}, the potential candidates for the rotation angle are of the type $\frac{2\pi}{k}$, where $k\leq n$. In particular, $k$ can take all admissible values except of odd $n$ in which case $k<n$.
\end{remark}

\begin{figure}[ht]
\begin{center}
  \includegraphics[width=0.4\textwidth]{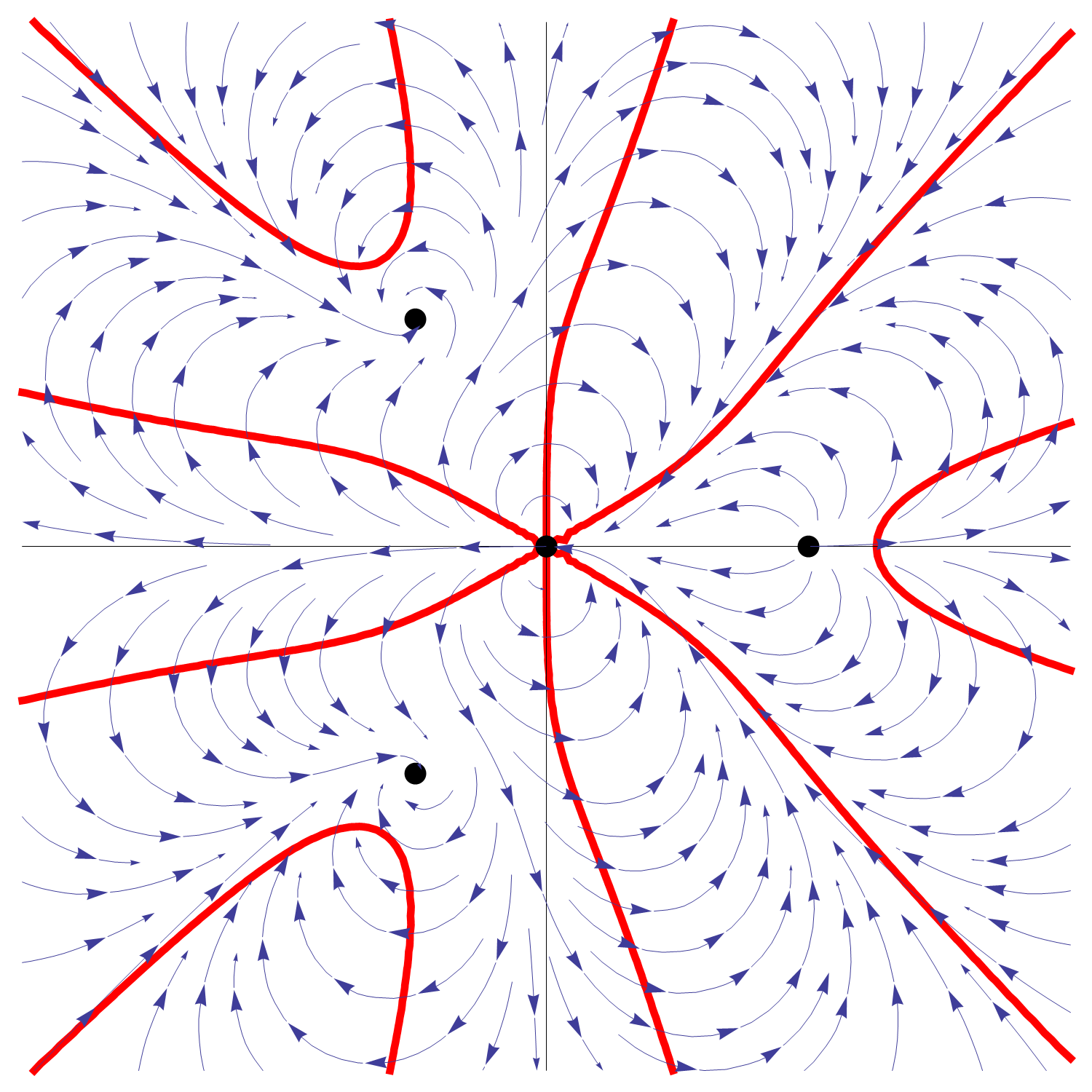}
\begin{minipage}{0.7\textwidth}
\caption{Curve $h(x,y)=0$ (red), Vector field $(\partial_x h,-\partial_y h)$ (blue) and its singularities (thick black dots). \label{fig:barycenter}}
\end{minipage}
\end{center}
\end{figure}

\begin{example}\label{sextic}
We start with a harmonic sextic
\begin{equation}
  h(x,y)=x^6-15 x^4 y^2-2 x^3+15 x^2 y^4+6 x y^2-y^6.
\end{equation}
The complex function associated with the vector field then has the form
\begin{equation}
g(x,y)=-2 x^3 + x^6 + 6 x y^2 - 15 x^4 y^2 + 15 x^2 y^4 - y^6 -\I \left(-30 x^4 y+60 x^2 y^3+12 x y-6 y^5\right),
\end{equation}
which after the substitution in Eq. \eqref{eq:subst} gives rise to
\begin{equation}
 g(z)=6z^5-6z^2.
\end{equation}
Therefore, $g(z)=z^2(6z^3-6)$, so $\tilde{g}(z)=6z^3-6$, and $N=3$. Formula \eqref{eq:center harmonic} reveals that the only possible center of rotational symmetry of the curve $h(x,y)=0$ is the point $\f p=0+0\I\simeq(0,0)$. The only non-trivial rotation angle for the curve (see Fig.~\ref{fig:barycenter}) is $\varphi=2\pi/3$; therefore $\sym^+(h)= \langle\rho_{\f p,2\pi/3}\rangle$.
\end{example}

Now let us focus on finding the candidates for the axes of symmetry, under the same assumptions as before. By Lemma~\ref{lem:center} and Theorem~\ref{thm:barycenter} we know that all the axes of symmetry intersect at the point $\f p =-\frac{a_{n-2}}{(n-1)a_{n-1}}$. If we move the point $\f p$ to the origin, the polynomial $h$ remains harmonic and the second highest coefficient of the new associated polynomial $g(z)$ vanishes. Then we have the following result.

\begin{theorem}\label{thm:direction of axis}
  Let $h(x,y)$ be a harmonic polynomial such that the associated polynomial $g(z)=\sum_{i=0}^{n-1}a_iz^i$ satisfies $a_{n-2}=0$.
  Then any possible symmetry axis of $h(x,y)=0$ passes through the origin. Let $\ell$ be the index such that $a_\ell\not=0$ but $a_i=0$ for each $i<\ell$.  If $\ell<n-1$ then the angles $\varphi$ between the tentative symmetry axes of $h(x,y)=0$ and the $x$-axis satisfy the~relationship
\begin{equation}\label{eq:direction harmonic}
      (n-\ell-1)\varphi\equiv \mathrm{arg}\left((-1)^{n-\ell-1}\frac{a_\ell}{a_{n-1}}\right) \mod 2\pi.
  \end{equation}
  Otherwise $g(z)=a_{n-1}z^{n-1}$ and the tentative axes of symmetry of $h(x,y)=0$ possess the directions
  \begin{equation}\label{eq:direction harmonic-2}
      \varphi = \frac{1}{n}\bigg( k\pi -\mathrm{arg}(a_{n-1})\bigg)
        \end{equation}
  for $k=1,\ldots,n$.
\end{theorem}

\begin{proof}
The fact that any possible axis of symmetry passes through the origin has been shown before.
Consider first the case $\ell<n-1$. If  $h(x,y)=0$ is symmetric w.r.t. some axis, then so are the roots of the associated polynomial $g(z)$. Since the axis passes through the origin, there are three types of roots: (1) the origin; (2) points on the axis distinct from the origin; and (3) roots not on the axis, paired with a second symmetric root.  Hence the polynomial $g$ can be factored as
\begin{equation}
   g(z)=\sum_{i=0}^{n-1}a_iz^i=a_{n-1}z^{\ell-1}\cdot\prod_{i=1}^{k_1}(z-c_i\e^{\I\varphi})\cdot\prod_{j=1}^{k_2}(z-d_j\e^{\I(\varphi+\psi_j)})(z-d_j\e^{\I(\varphi-\psi_j)}).
\end{equation}
The product of the nonzero roots of $g(z)$ is equal to
\begin{equation}
  \left(\prod_{i=1}^{k_1}c_i\right)\cdot\left(\prod_{j=1}^{k_2}d_j^2\right)\e^{\I(k_1+2k_2)\varphi},
\end{equation}
where $k_1+2k_2=n-\ell-1$. Then the relation \eqref{eq:direction harmonic} follows from Cardano-Vieta's formulae.

It remains to prove the case $g(z)=a_{n-1}z^{n-1}$. Now directly by Lemma \ref{G-z} we have
\begin{equation}
  h(x,y)=\re\left(\int g(z)\diff z\right) = \re\left(\frac{a_{n-1}}{n}z^n+\bf{c}\right)=h^{\star}(x,y)+c,
\end{equation}
where $h^{\star}=\re\left(\frac{a_{n-1}}{n}z^n\right)$ and $c=\re (\bf{c})$. Writing $z^n =r^n(\cos n\varphi+\I \sin n\varphi)$ we immediately see that $h^{\star}(x,y)=0$ consists of $n$ lines through the origin, forming angles $-\mathrm{arg}(a_{n-1})+k\frac{\pi}{2n}$ with the $x$-axis, where $0\leq k<n$. Furthermore, since $h^{\star}(x,y)$ corresponds to an arrangement of lines through the origin, $h^{\star}$ is homogeneous, in fact the homogeneous form of $h(x,y)$ of maximum degree. Finally, by Remark \ref{high-form} and since we know that all the axes of symmetry of $h(x,y)=0$ go through the origin, we conclude that the axes of symmetry of $h(x,y)=0$ are among the axes of symmetry of $h^{\star}(x,y)=0$.
\end{proof}

\begin{example}\label{sextic}
We consider again the sextic curve in Example \ref{sextic}. Here $g(z)$ already has the form required in Theorem \ref{thm:direction of axis}. Additionally, $n-1=5$, $\ell=2$, $a_{n-1}=6$, $a_\ell=-6$. Therefore, Eq. \eqref{eq:direction harmonic} provides \[3\varphi \equiv 0 \mod 2\pi.\]
Therefore, the only possible symmetry axes are the straight lines passing through the origin and enclosing the angles $0,2\pi/3,4\pi/3$ with the $x$-axis, which are certainly symmetry axes of the curve.
\end{example}

\begin{remark} The symmetry group of $h$ can be either equal to the symmetry group of the $n$-gon, or strictly a subgroup of it. For instance, both $h_1(x,y)=x^2-y^2$ and $h_2(x,y)=x^2-y^2+1$ give rise to $g(z)=2z$. However, the symmetry group of $h_1$ is the symmetry group of a square, $D_4$, while the symmetry group of $h_2$ is a proper subgroup of the symmetry group of the square, not including the mirror symmetries with respect to the lines $y=\pm x$.
\end{remark}

The main steps to compute the symmetries of a harmonic polynomial are summarized in Algorithm 1.

\begin{algorithm}[t!]
\begin{algorithmic}[1]
\REQUIRE A harmonic polynomial $h(x,y)$ defining an algebraic curve $\av{C}$, with finitely many symmetries.
\ENSURE The symmetries of $\av{C}$.
\STATE{[Rotation symmetries]}
\STATE{Compute the tentative center of rotation ${\bf p}$ of $\av{C}$ by means of the expression Eq. \eqref{eq:center harmonic}. }
\FOR{each angle $\varphi=\frac{2\pi}{k}$, with $k\leq\mbox{deg}(\av{C})$}
\STATE{check whether $h\circ \rho_{{\bf p},\varphi}=\pm h$, where $\phi_{{\bf p},\varphi}$ is the rotation of center ${\bf p}$ and rotation angle $\varphi$.}
\ENDFOR
\STATE{[Reflections]}
\STATE{Apply a translation $\tau$ such that $\tau({\bf p})=(0,0)$; let $h:=h\circ \tau^{-1}$}
\STATE{Compute the polynomials $g(x,y)$ and $g(z)$ in Eq. \eqref{assoc-g} and Eq. \eqref{assoc-z}.}
\IF{$g(z)\neq a_{n-1}z^{n-1}$}
\FOR{the $\varphi$s as in Eq. \eqref{eq:direction harmonic}}
\STATE{check whether $h\circ \rho_{{\bf 0},\varphi}=\pm h$, where $\rho_{{\bf 0},\varphi}$ is the reflection on the line through the origin, forming an angle $\varphi$ with the $x$-axis.}
\ENDFOR
\ELSE
\FOR{the $\varphi$s as in Eq. \eqref{eq:direction harmonic-2}}
\STATE{check whether $h\circ \rho_{{\bf 0},\varphi}=\pm h$, where $\rho_{{\bf 0},\varphi}$ is the reflection on the line through the origin, forming an angle $\varphi$ with the $x$-axis.}
\ENDFOR
\ENDIF
\STATE{{\bf return} the symmetries found, or the message {\tt No symmetries found}.}
\end{algorithmic}
\caption{Symms-harmonic}\label{alg-syms-harmoni}
\end{algorithm}

\subsection{Reduction to the harmonic case.}\label{red-symmetries}

Now let $f(x,y)$ be a square-free polynomial, not necessarily harmonic, defining a curve $\av{C}$. We want to find the symmetries of $\av{C}$. Successive application of the Laplace operator yields the sequence
\begin{equation}\label{eq:laplacianchain}
  f\longmapsto\triangle f\longmapsto\triangle^2 f\longmapsto\cdots\longmapsto\triangle^\ell f\longmapsto c,
\end{equation}
where $c$ is a constant (possibly zero), and the corresponding chain of groups of symmetries
\begin{equation}
  \sym(f)\subset\sym(\triangle f)\subset\sym(\triangle^2 f)\subset\cdots\subset\sym(\triangle^\ell f)\subset\iso{2}.
\end{equation}

We will refer to the chain in Eq. \eqref{eq:laplacianchain} as a \emph{chain of laplacians}. Now depending on the degree of $\triangle^\ell f$ we distinguish the following cases:

\vspace{0.3 cm}

\noindent\fbox{$\deg\triangle^\ell f=n>2$.} In this case $\triangle^\ell f$ is harmonic and we can use previous results to find $\sym(\triangle^\ell f)$. Then the symmetries of the original curve form a subgroup of $\sym(\triangle^\ell f)$ that can be easily identified.

\begin{example}
Let ${\mathcal C}$ be the planar algebraic curve defined by $f(x,y)=0$, where $f(x,y)=\frac{1}{10}x^5-\frac{1}{2}x^3y^2+\frac{1}{2}x^2+1$ (see Fig. \ref{fig:simmetry}). Here we get
\[
\triangle f=h=x^3-3xy^2+1,\mbox{ }\triangle^2 f=\triangle h=0.
\]
Therefore, $\sym(f)\subset \sym(h)$. In this case the complex function $g(z)$ associated with $h(x,y)$ is $g(z)=3z^2$. Hence, the possible center of rotation is the origin, the possible rotation angle is $\pi$, and the possible reflection axes are the coordinate axes. Regarding rotational symmetry,
\[f(-x,-y)=-\frac{1}{10}x^5+\frac{1}{2}x^3y^2+\frac{1}{2}x^2+1,\]
which is not a multiple of $f(x,y)$; therefore, ${\mathcal C}$ does not have rotational symmetry. Regarding reflections, $f(-x,y)$ is not a multiple of $f(x,y)$ either, but $f(x,-y)=f(x,y)$. Hence, ${\mathcal C}$ is symmetric only with respect to the $x$-axis.

\begin{figure}
\begin{center}
\includegraphics[scale=0.3]{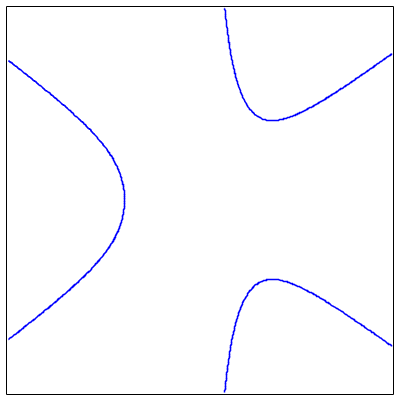}
\caption{The curve $\frac{1}{10}x^5-\frac{1}{2}x^3y^2+\frac{1}{2}x^2+1=0$.}\label{fig:simmetry}
\end{center}
\end{figure}

\end{example}

\vspace{0.3 cm}
\noindent\fbox{$\deg\triangle^\ell f=1$.} In this case $\triangle^\ell f$ defines a line $ax+by+c=0$. Since the symmetry group of a line is infinite, we need to reduce the possible candidates; in fact, our strategy will be to find a new bivariate, \emph{harmonic} polynomial $h$ defining a curve $\av{\tilde{C}}$ such that $\sym(\av{C})$ is a proper subgroup of $\sym(\av{\tilde{C}})$. In order to do this, first we observe that since the laplacian operator commutes with orthogonal transformations, applying to $\av{C}$ an orthogonal transformation ${\mathcal T}$ transforming $ax+by+c=0$ into $x=0$, we get $\triangle^\ell f=\tilde{a} x$. In general we cannot guarantee that $\tilde{a}=1$, but if $\tilde{a}\neq 1$ the reasoning can be easily adapted. So for simplicity we will assume that $\triangle^\ell f= x$.

Now let $k\leq \ell$ be the smallest natural number such that $\triangle^k f=p(x)$ only depends on $x$, so that $\widehat{f}=\triangle^{k-1} f$ \emph{does} depend on both $x,y$; notice that it can happen that $k=\ell$. Also, let $q\geq 0$ such that $k+q=\ell$, and let $\tilde{p}(x)\in \R[x]$ satisfy that $\tilde{p}''(x)=p(x)$ and $\tilde{p}'(0)=\tilde{p}(0)=0$, where $\tilde{p}''(x)$ represents the second derivative of $\tilde{p}(x)$; thus, seeing $p,\tilde{p}$ as functions of $x,y$, we have $\Delta \tilde{p}=p$. Also, observe that if we write
\begin{equation}\label{exp-p}
p(x)=x^{2q+1}+a_{2q}x^{2q}+\cdots +\cdots,
\end{equation}
then
\begin{equation}\label{p-tilde}
\tilde{p}(x)=\frac{x^{2q+3}}{(2q+3)(2q+1)}+a_{2q}\frac{x^{2q+1}}{(2q+2)(2q+1)}+\cdots
\end{equation}
Calling $h(x,y)=x$, the following diagrams show that $\widehat{f},p,h$, on one hand, and $\tilde{p},p,h$, on the other hand, belong to two chains of laplacians:

\begin{equation}\label{eq:laplacianchain2}
  \widehat{f}\longmapsto p \longmapsto \stackrel{\underbrace{q}}{\mbox{ }\cdots\mbox{ }}\longmapsto h
\end{equation}

\begin{equation}\label{eq:laplacianchain3}
  \tilde{p}\longmapsto p \longmapsto \stackrel{\underbrace{q}}{\mbox{ }\cdots\mbox{ }}\longmapsto h.
\end{equation}

Additionally, $\widehat{f}$ belongs to the laplacian chain starting with $f$. Now if $\phi$ is a symmetry of $f$, by Theorem \ref{thm:symmetry preservation} we know that $\phi$ is also a symmetry of $\widehat{f}$. Hence by Lemma \ref{lem:lambda} we have $\widehat{f}\circ \phi=\alpha\cdot \widehat{f}$, where $\alpha=1$ or $\alpha=-1$. Furthermore, since the laplacian commutes with orthogonal transformations, we also have $p\circ \phi=\alpha \cdot p$, and $h\circ \phi=\alpha \cdot h$; in particular, the positive or negative value of $\alpha$ is consistently the same for $\widehat{f},p,h$. Now let $h=\widehat{f}-\tilde{p}$; then $\triangle h=\triangle \widehat{f}-\triangle \tilde{p}=p-p=0$.

\begin{proposition} Let $h=\widehat{f}-\tilde{p}$, where $\tilde{p}=\tilde{p}(x)$ and $\widehat{f}=f(x,y)$ are defined as above. Then $\sym(f)\subset \sym(h)$.
\end{proposition}

\begin{proof} It suffices to see that $\sym(\widehat{f})\subset \sym(h)$. In order to prove this, let us show that for any $\phi\in\sym(\widehat{f})$, $\widehat{f}\circ \phi=\alpha \cdot \widehat{f}$ implies $h\circ \phi=\alpha \cdot h$ too. Writing $\phi(x,y)=(\phi_1(x,y),\phi_2(x,y))$ and since $h\circ \phi=\alpha \cdot h$, $\alpha=1$ implies $\phi_1(x,y)=x$, and $\alpha=-1$ implies $\phi_1(x,y)=-x$. Let us analyze both cases separately. If $\alpha=1$, since $\phi_1(x,y)=x$ we trivially get $\tilde{p}\circ \phi=\alpha\cdot \tilde{p}$. Therefore, \[h\circ \phi=(\widehat{f}-\tilde{p})\circ \phi=\widehat{f}\circ \phi-\tilde{p}\circ \phi=\alpha\cdot \widehat{f}-\alpha \cdot \tilde{p}=\alpha \cdot h.\]If $\alpha=-1$, since $\phi_1(x,y)=-x$ and $p\circ \phi=\alpha \cdot p$, we deduce that $p(-x)=-p(x)$, so $p(x)$ consists only of terms of odd degree. Therefore $\tilde{p}$ also consists only of terms of odd degree, and $\tilde{p}(-x)=-\tilde{p}(x)$ too, i.e. $\tilde{p}\circ \phi=\alpha \cdot \tilde{p}$. Then the result follows as before.
\end{proof}

Since $\widehat{f}$ depends on $y$, $h$ has degree at least 1 in $y$. If $\deg h>2$ then we are in the case $\deg\triangle^\ell f=n>2$. If $\deg h=2$ then since $h$ is harmonic, by Corollary \ref{corol-finite} $\sym(h)$ is finite, and we know that $\sym(f)\subset \sym(h)$. If $\deg h=1$ then $h\neq kx$, with $k$ a constant, because $h$ depends explicitly on $y$; thus, the symmetries of $f$ are among the finitely many, simultaneous symmetries of the conic $x\cdot h(x,y)=0$. We summarize the ideas for the case $\deg\triangle^\ell f=1$ in Algorithm 2. Here we also show how to proceed in the case when $\triangle^\ell f$ is a general linear polynomial $\xi(x,y)$, not necessarily $\xi(x,y)=x$. Notice that Step 4 of Algorithm 2 amounts to solving a linear system of equations in the coefficients of $\tilde{p}(x,y)$.

\begin{algorithm}[t!]
\begin{algorithmic}[1]
\REQUIRE A square-free polynomial $f(x,y)$ defining a real algebraic curve $\av{C}$, not a union of concurrent lines or a union of concentric circles, such that the case $\deg\triangle^\ell f=1$ occurs.
\ENSURE A polynomial $h$ such that $\sym(f)\subset \sym(h)$, which is either harmonic, or a conic, where $\sym(h)$ is finite.
\STATE{let $\xi(x,y)=\Delta^{\ell}f$ [notice that $\xi(x,y)$ is linear in $x,y$.]}
\STATE{let $k$ be the smallest number such that $\Delta^k f$ is the composition of a univariate function, and $\xi(x,y)$}
\STATE{let $\widehat{f}:=\Delta^{k-1}f$}
\STATE{let $\tilde{p}(x,y)$ be the polynomial of degree $\deg(\Delta^k f)+2$ such that: (i) $\Delta \tilde{p}=\Delta^k f$; (ii) $\tilde{p}(0,0)=0$; (iii) $\tilde{p}_x(0,0)=\tilde{p}_y(0,0)=0$}
\STATE{let $h:=\widehat{f}-\tilde{p}$ [notice that $h$ is harmonic.]}
\IF{$\deg(h)=1$}
\STATE{$h:=h\cdot \xi(x,y)$. [this polynomial is not necessarily harmonic, but defines a conic with finitely many symmetries, containing $\sym(f)$].}
\ENDIF
\end{algorithmic}
\caption{Case $\deg\triangle^\ell f=1$}\label{case-1}
\end{algorithm}

\begin{example}
Let $f(x,y)=\frac{1}{3}x^3-y+1$. Here $\triangle f=x$. Since $\tilde{p}(x)=\frac{1}{3}x^3$, we get $\widehat{f}=f-\tilde{p}=-y+1$. Therefore, the possible symmetries of $f$ are the simultaneous symmetries of $x=0$ and $y=1$. However, one can check that none of these are, in fact, symmetries of $f(x,y)=0$.
\end{example}

\vspace{0.3 cm}
\noindent\fbox{$\deg\triangle^\ell f=2$.} In this case, $\triangle^{\ell}f(x,y)=0$ is a conic section. If the conic is not a pair of parallel lines or a circle, we can easily find its axes of symmetry and the center of symmetry by using elementary, Linear Algebra methods. The case of two parallel lines (possibly a double line) can be handled as the case $\deg\triangle^\ell f=1$. However, if we get a circle we already know the possible center of rotation but we have no information about the directions of the symmetry axes. So let us address this case; the strategy is analogous to the case $\deg\triangle^\ell f=1$.

Without loss of generality we can assume that the center of the circle is the origin. Let $p=\triangle^k f$ be the first polynomial in the chain of laplacians starting with $f$, leading to concentric circles; therefore, $\widehat{f}=\triangle^{k-1} f$ is not a union of concentric circles.  Notice that $p(x,y)=\prod_i(x^2+y^2-\rho_i)$, so changing to polar coordinates we have $p=p(r)$, and in fact $p=p(r^2)$. Similarly to the case $\deg\triangle^\ell f=1$, we seek another polynomial $\tilde{p}(x,y)$ such that $\tilde{p}=\tilde{p}(r)$ and $\triangle\tilde{p}=p$. In order to do this, we use Eq. \eqref{eq:laplace polar}, taking into account that $\tilde{p}$ does not depend on $\theta$. Then we have
\[
\triangle \tilde{p}=\frac{1}{r}\cdot \left(r\cdot \tilde{p}''+\tilde{p}'\right)=\frac{1}{r}\cdot \frac{d}{dr}\left(r\cdot \tilde{p}'\right)=p,
\]
so we need to solve
\[
\frac{d}{dr}\left(r\cdot \tilde{p}'\right)=r\cdot p.
\]
Let $P(r)$ be a primitive of $r\cdot p$ such that $P(0)=0$; in fact, since $p=p(r^2)$, one has $P=P(r^2)$ too. Then $\tilde{p}$ is a primitive of $\frac{1}{r}\cdot P$. Notice that $\tilde{p}=\tilde{p}(r^2)$ as well, and that the integration constant in $\tilde{p}$ can be chosen so that $\tilde{p}(0)=\tilde{p}'(0)=0$.

Finally, let $h=\widehat{f}-\tilde{p}$. Then $\triangle h=\triangle \widehat{f}-\triangle \tilde{p}=p-p=0$.

\begin{proposition}\label{prop-r} Let $h=\widehat{f}-\tilde{p}$, where $\tilde{p}=\tilde{p}(r)$, defined as above. With the preceding assumptions, $\sym(f)\subset \sym(h)$.
\end{proposition}

\begin{proof} It suffices to see that $\sym(\widehat{f})\subset \sym(h)$. Let $\phi$ be a symmetry of $\widehat{f}$. Since we are assuming that the center of all the irreducible components of the curve defined by $\widehat{f}$ is the origin, we have $\widehat{f}=\widehat{f}(x^2+y^2)$. Since any orthogonal transformation leaves the form $x^2+y^2$ invariant, we have $\widehat{f}\circ \phi=\widehat{f}$. Additionally, since $\tilde{p}=\tilde{p}(r^2)=\tilde{p}(x^2+y^2)$, we also have $\tilde{p}\circ \phi=\tilde{p}$. Therefore,
\[
h\circ \phi=(\widehat{f}-\tilde{p})\circ \phi=\widehat{f}\circ \phi-\tilde{p}\circ \phi=\widehat{f}-\tilde{p}=h.
\]
\end{proof}

As in the case $\deg\triangle^\ell f=1$, either $\deg h>1$, in which case $\sym(h)$ is finite, or $\deg h=1$. However, in this last case the symmetries of $f$ are among the simultaneous symmetries of the circle  $\triangle^{\ell}f(x,y)=0$ and $h(x,y)=0$, whose computation is straightforward. We summarize the main ideas of the case $\deg\triangle^\ell f=2$ in Algorithm 3.

\begin{algorithm}[t!]
\begin{algorithmic}[1]
\REQUIRE A square-free polynomial $f(x,y)$ defining a real algebraic curve $\av{C}$, not a union of concurrent lines or a union of concentric circles, such that the case $\deg\triangle^\ell f=2$ occurs.
\ENSURE A polynomial $h$ such that $\sym(f)\subset \sym(h)$ which is either harmonic, or a conic, or the product of a line and a conic, where $\sym(h)$ and is finite.
\IF{$\Delta^\ell f$ does not correspond to either a circle, or a double line, or two parallel lines}
\STATE{$h:=\Delta^{\ell} f$ [this polynomial can be harmonic, or not]}
\ENDIF
\IF{$\Delta^{\ell} f=a(\xi(x,y)+c_1)(\xi(x,y)+c_2)$, with $a,c_1,c_2\in {\Bbb R}$}
\STATE{proceed as in the case $\deg(\Delta^\ell f)=1$ to compute $h$}
\ENDIF
\IF{$\Delta^{\ell} f$ defines a circle $C$}
\STATE{apply a translation $\tau$ such that the center of $C$ is mapped to $(0,0)$; let $f:=f\circ \tau^{-1}$}
\STATE{let $k$ be the smallest number such that $\Delta^k f=p(x^2+y^2)$ is a union of circles centered at $(0,0)$.}
\STATE{let $\widehat{f}:=\Delta^{k-1}f$}
\STATE{let $\tilde{p}(r)$ be the polynomial satisfying: (i) $\tilde{p}=\tilde{p}(r)$; (ii) $\Delta \tilde{p}=p$ [where $\Delta p$ is computed in polar coordinates]; (iii) $p(r)=p'(r)=0$ [observe that in fact $\tilde{p}=\tilde{p}(r^2)=\tilde{p}(x^2+y^2)$]}
\STATE{let $h:=\widehat{f}-\tilde{p}$ [notice that $h$ is harmonic.]}
\IF{$\deg(h)=1$}
\STATE{$h:=h\cdot \Delta^{\ell} f$ [this polynomial can be harmonic, or not]}
\ENDIF
\ENDIF
\end{algorithmic}
\caption{Case $\deg\triangle^\ell f=2$}\label{case-2}
\end{algorithm}

Finally, the main steps to find the symmetries of a given curve $f(x,y)=0$ by using the ideas in the section are summarized in the algorithm {\tt Symm-General}.

\begin{algorithm}[t!]
\begin{algorithmic}[1]
\REQUIRE A square-free polynomial $f(x,y)$ defining a real algebraic curve $\av{C}$, not a union of concurrent lines or a union of concentric circles.
\ENSURE The symmetries of $\av{C}$.
\FOR{$s=1$ to $\lceil \mbox{deg}(\av{C})/2 \rceil $}
\STATE{compute $\Delta^s f$}
\IF{$\deg(\Delta^s f)=0$}
\STATE{$\ell:=s-1$; {\bf break}}
\ENDIF
\ENDFOR
\IF{$\deg(\Delta^\ell f)>2$}
\STATE{$h:=\Delta^\ell f$}
\ENDIF
\IF{$\deg(\Delta^\ell f)=1$}
\STATE{find the polynomial $h$ using Algorithm 2}
\ENDIF
\IF{$\deg(\Delta^\ell f)=2$}
\STATE{find the polynomial $h$ using Algorithm 3}
\ENDIF
\STATE{compute the symmetries of $h$}
\FOR{each symmetry $\phi$ found in the previous step:}
\STATE{check whether $f\circ \phi=\pm f$}
\ENDFOR
\STATE{{\bf return} the symmetries of $f$, or the message {\tt No symmetries found}.}
\end{algorithmic}
\caption{{\tt Symm-General}}\label{alg-syms-general}
\end{algorithm}

In particular, throughout the section we have proven the following result.

\begin{corollary}\label{th-realize-2}
Let  ${\mathcal C}$ be an algebraic curve with a non-trivial, finite group of symmetry $\sym(\av{C})$, and let $f(x,y)$ be a square-free polynomial implicitly defining ${\mathcal C}$, with coefficients in a field ${\Bbb K}\supset {\Bbb Q}$. Then the coordinates of the center of symmetry, if any, are elements of ${\Bbb K}$.
\end{corollary}

\section{Computation of the similarities.}\label{sec-similarities}

Let ${\mathcal C}_1,{\mathcal C}_2$ be two algebraic planar curves, defined by square-free polynomials $f_1,f_2$, of the same degree $n$; notice that if the degrees of $f_1,f_2$ are different, the corresponding curves cannot be similar. Furthermore, we assume that ${\mathcal C}_1,{\mathcal C}_2$ are not both unions of parallel lines or of concentric circles; therefore there are just finitely many similarities relating them (see Theorem 3 in \cite{AHM14b}). In order to check whether or not ${\mathcal C}_1,{\mathcal C}_2$ are similar, and to compute the similarities relating them in the affirmative case, we will follow a strategy analogous to the preceding section. Thus, we start assuming that $f_1,f_2$ are harmonic polynomials, and then we show how to reduce any other case, to this case. As in the preceding section, our method computes finitely many similarities, that can be tested afterwards by using Eq. \eqref{eq:test of similarity}.

\subsection{Similarities of harmonic polynomials}\label{subsec-harmonic-sims}

Let $f_1,f_2$ be harmonic of degree $n\geq 3$ such that ${\mathcal C}_1,{\mathcal C}_2$ are related by a similarity $\phi$; the cases $n=1$ and $n=2$ can be solved by elementary methods and are briefly addressed later. Let $\overrightarrow{v}_1(x,y)=(\partial_x f_1,-\partial_y f_1)$, and $\overrightarrow{v}_2(x,y)=(\partial_x f_2,-\partial_y f_2)$. By Lemma \ref{lem-recognize-2}, we have $f_1\circ \phi=\lambda  f_2$, and therefore $\phi$ maps the {real} singular points of the vector field $\overrightarrow{v}_1$ onto the {real} singular points of the vector field $\overrightarrow{v}_2$, and conversely. Let
\[
g_1(z)=b_{n-1}z^{n-1}+b_{n-2}z^{n-2}+\cdots +b_0,\mbox{ }g_2(z)=c_{n-1}z^{n-1}+c_{n-2}z^{n-2}+\cdots+c_0,
\]
with $b_{n-1}\neq 0,c_{n-1}\neq 0$, be the complex polynomials associated with $f_1,f_2$ as in Eq. \eqref{assoc-g}; then $\phi$ maps roots of $g_1(z)$ to roots of $g_2(z)$, and conversely. As we observed in Subsection \ref{subsec-simil}, we can identify $\phi$ with a linear complex transformation $\phi(z)=\alpha z+\beta$ or $\phi(z)=\alpha\overline{z}+\beta$, with $\alpha\neq 0$. Assume first that $\phi(z)=\alpha z+\beta$. Then we have
\begin{equation}\label{cond-sim}
g_1(\alpha z+\beta)=\lambda g_2(z)
\end{equation}
for some nonzero $\lambda \in {\Bbb C}$. The condition in Eq. \eqref{cond-sim} gives rise to a triangular system with $n$ equations
\begin{equation}\label{system}
\left\{\begin{array}{lcr}
b_{n-1}\alpha^{n-1} & = & \lambda c_{n-1},\\
b_{n-1}\alpha^{n-2}(n-1)\beta + b_{n-2}\alpha^{n-2} & = & \lambda c_{n-2},\\
\cdots & = & \cdots
\end{array}\right.
\end{equation}

From the first two equations and since $\alpha\neq 0$ we have
\begin{equation}\label{variables}
\lambda=\frac{b_{n-1}\alpha^{n-1}}{c_{n-1}},\mbox{ }\beta=\frac{1}{(n-1)b_{n-1}}\cdot\left(\frac{b_{n-1}c_{n-2}}{c_{n-1}}\alpha-b_{n-2}\right).
\end{equation}
Plugging these expressions in the remaining $n-2$ equations, we obtain $n-2$ polynomials in the variable $\alpha$, with complex coefficients. Let $M(\alpha)$ be the gcd of these (univariate, complex) polynomials; such a gcd can be fastly and efficiently computed, for instance, using the computer algebra system {\tt Maple} 18. Then the following result holds.

\begin{proposition} Let $f_1,f_2$ be two harmonic polynomials of degree $n\geq 3$, and let ${\mathcal C}_1,{\mathcal C}_2$ be the planar algebraic curves defined by $f_1,f_2$. The direct similarities relating ${\mathcal C}_1,{\mathcal C}_2$ are $\phi(z)=\alpha_i z+\beta_i$, where $\alpha_i$ is a nonzero root of $M(\alpha)$, and $\beta_i$ corresponds to Eq. \eqref{variables}. In particular, ${\mathcal C}_1,{\mathcal C}_2$ are related by a direct similarity iff $M(\alpha)$ has some nonzero root.
\end{proposition}

For opposite similarities, we have $\phi(z)=\alpha\overline{z}+\beta$. Let ${\bf g}_2(z)$ be the polynomial whose coefficients are the conjugates of the coefficients of $g_2(z)$. Then
\begin{equation}\label{cond-sim2}
g_1(\alpha z+\beta)=\lambda {\bf g}_2(z),
\end{equation}
and we can proceed as in the case of direct similarities. The whole procedure, both for direct and opposite similarities, is given in Algorithm 5.

\begin{algorithm}[t!]
\begin{algorithmic}[1]
\REQUIRE Two harmonic polynomials $f_1(x,y),f_2(x,y)$ of the same degree.
\ENSURE The similarities between the curves $\av{C}_1,\av{C}_2$ defined by $f_1,f_2$.
\IF{$\deg(f_1)=\deg(f_2)\leq 2$}
\STATE{compute the similarities, if any, by elementary methods; and {\bf return} either the similarities found, or the message {\tt The curves are not similar}}
\ENDIF
\STATE{compute the polynomials $g_1(z),g_2(z)$}
\STATE{compute the system Eq. \eqref{system} in $\lambda,\alpha,\beta$ from Eq. \eqref{cond-sim} [direct similarities]}
\STATE{check whether the system has any solutions $(\lambda,\alpha,\beta)$ with $\alpha\neq 0$}
\STATE{{\bf return} the direct similarities $\phi(z)=\alpha z +\beta$ found, or the message {\tt No direct similarities found}}
\STATE{compute the system in $\lambda,\alpha,\beta$ derived from Eq. \eqref{cond-sim2} [opposite similarities]}
\STATE{check whether the system has any solutions $(\lambda,\alpha,\beta)$ with $\alpha\neq 0$}
\STATE{{\bf return} the direct similarities $\phi(z)=\alpha \overline{z} +\beta$ found, or the message {\tt No opposite similarities found}}
\end{algorithmic}
\caption{{\tt Similarities Harmonic}}\label{sim-case-1}
\end{algorithm}

If $f_1,f_2$ have degree one, the infinitely many similarities between ${\mathcal C}_1$ and ${\mathcal C}_2$ can be written as
\[
\phi_{\gamma, \delta}=\gamma e^{\I \theta}z+{\bf v}+\delta {\bf w},\mbox{ or }\phi_{\gamma, \delta}=\gamma e^{\I \theta}\overline{z}+{\bf v}+\delta {\bf w},
\]
where $\gamma,\delta\in {\Bbb R}$, $\theta$ is the angle between ${\mathcal C}_1$ and ${\mathcal C}_2$, ${\bf v}$ is any vector connecting two points of ${\mathcal C}_1$ and ${\mathcal C}_2$, and ${\bf w}$ represents a vector parallel to ${\mathcal C}_2$. If $f_1,f_2$ have degree two, we can derive the nature of ${\mathcal C}_1,{\mathcal C}_2$ from its matrix representations, and then write their canonical equations. Then ${\mathcal C}_1,{\mathcal C}_2$ are similar iff the canonical equations are multiples of each other. Furthermore, in the affirmative case, in order to compute the similarities between ${\mathcal C}_1,{\mathcal C}_2$ we compute the affine transformation mapping the coordinate systems where the canonical representations are written; composing this mapping with the symmetries of ${\mathcal C}_1,{\mathcal C}_2$, all the similarities between ${\mathcal C}_1,{\mathcal C}_2$ are derived (see Proposition 2 in \cite{AHM14b}). Notice that this strategy works for conics in general, regardless of whether or not they are harmonic.

\begin{example}
Let $f_1(x,y)$ be the polynomial defining the sextic in Example \ref{sextic}, and let $f_2(x,y)$ be the result of applying the transformation
\[
x:=-y+2,\mbox{ }y:=x+1.
\]
As we saw in Example \ref{sextic}, $g_1(z)=6z^5-6z^2$. Furthermore, \[g_2(z)=234z+660z^2+60\I z^4+240\I z^3-114\I z^2-708\I z+180 z^3-30 z^4-6 z^5-222-246\I.\]For direct similarities, after imposing the condition in Eq. \eqref{cond-sim} the resulting system has three solutions with $\alpha\neq 0$, where $\alpha$ is any cubic root of $-\I$, and $\beta=(1-2\I)\alpha $. Hence, we get three direct similarities; notice that getting more than one direct similarity is expectable, since the curves have direct symmetries. As for opposite similarities, the system derived from Eq. \eqref{cond-sim2} has three solutions with $\alpha\neq 0$, where $\alpha$ is any cubic root of $\I$, and $\beta=(1+2\I)\alpha$. Again, getting more than one direct similarity is expectable, since the curves have opposite symmetries.
\end{example}

\begin{remark}\label{efficiency-5} About the efficiency of Algorithm 5, the fact that the system to be solved is triangular guarantees that Algorithm 5 works very fast. For instance, checking whether two algebraic curves ${\mathcal C}_1$ and ${\mathcal C}_2$, of degree 20, with coefficients of order $2^{\alpha_1}$ and $2^{\alpha_2}$, $\alpha_1=27$ and $\alpha_2=52$, where ${\mathcal C}_2$ is built from ${\mathcal C}_1$ by applying a similarity, are similar, takes 0.436 seconds using Maple 18 on a personal laptop with a 2.90 GHz processor, and 8.00 Gbytes of RAM memory. The computation includes finding all the similarities between the curves.
\end{remark}

\subsection{Reduction to the harmonic case.}

Now let $f_1,f_2$ be square-free polynomials, not necessarily harmonic, defining curves $\av{C}_1,\av{C}_2$ related by some (unknown) similarity $\phi$. In order to find $\phi$, we consider the double laplacian chain
\begin{equation}\label{eq:laplacianchain2}
\begin{array}{l}
  f_1\longmapsto\triangle f_1\longmapsto\triangle^2 f_1\longmapsto\cdots\longmapsto\triangle^\ell f_1\longmapsto c_1\\
	f_2\longmapsto\triangle f_2\longmapsto\triangle^2 f_2\longmapsto\cdots\longmapsto\triangle^\ell f_2\longmapsto c_2,
\end{array}
\end{equation}
where $c_1,c_2\in {\Bbb R}$. Furthermore, by Theorem \ref{thm:symmetry preservation} we have the following chain of similarity groups,
\[\simm(f_1,f_2)\subset \simm(\triangle f_1,\triangle f_2)\subset \cdots \subset \simm(\triangle^\ell f_1,\triangle^\ell f_2)\subset \simm_2.\]
We also need the following result, which follows from Lemma \ref{Laplacian-similarities}.

\begin{lemma} \label{lambdas}
Let $f_1,f_2$ satisfy $f_1\circ \phi=\lambda f_2$, where $\phi\in \simm_n$ with scaling constant $\mu$. Then for $k\in {\Bbb N}$ we have
\begin{equation}\label{laplacian-k}
\triangle ^k f_1\circ \phi=\lambda_k \cdot \triangle^k f_2,
\end{equation}
where $\lambda_k=\frac{\lambda}{\mu^{2k}}$. In particular,
\begin{equation}\label{recursion}
\lambda_k=\frac{1}{\mu^2}\cdot \lambda_{k-1}.
\end{equation}
\end{lemma}

Now we distinguish the following cases depending on the degree of $\triangle^\ell f_1, \triangle^\ell f_2$.

\vspace{0.3 cm}

\noindent\fbox{$\deg\triangle^\ell f_1,\triangle^\ell f_2=n>2$.} In this case the $\triangle^\ell f_i$, $i=1,2$ are harmonic and we can use previous results to calculate $\simm(\triangle^\ell f_1,\triangle^\ell f_2)$.

\vspace{0.3 cm}
\noindent\fbox{$\deg\triangle^\ell f_1,\triangle^\ell f_2=1$.} The strategy is analogous to the symmetries case. Assume that $\Delta^\ell f_1=x$, and $\Delta^\ell f_2=ax+by+c$. Since $\Delta^\ell f_1\circ \phi=\lambda_\ell \cdot \Delta^\ell f_2$, calling $\phi(x,y)=(\phi_1(x,y),\phi_2(x,y))$, we have $\phi_1(x,y)=\lambda_\ell \cdot(ax+by+c)$. Now reasoning as in Subsection \ref{red-symmetries}, we define $k$ as the smallest natural number such that $\triangle^k f_1=p_1(x)$ only depends on $x$, and we denote $\widehat{f}_1=\triangle^{k-1}f_1$, so $\Delta \widehat{f}_1=p_1$. Also, let $\tilde{p}_1(x)$ satisfy $\tilde{p}''_1(x)=\Delta \tilde{p}_1=p_1$, and $\tilde{p}''_1(0)=\tilde{p}1'(0)=0$. Hence, $h_1=\widehat{f}_1-\tilde{p}_1$ is harmonic.

Additionally, since for $r=0,\ldots,\ell$ we have $\Delta^r f_1\circ \phi=\lambda_r\cdot \Delta^r f_2$, $k$ is also the smallest natural number such that $\triangle^k f_2$ only depends on $\phi_1(x,y)$, i.e. such that $\triangle^k f_2=p_2(x,y)$, where $p_2=\omega(\phi_1(x,y))=\omega\circ \phi$, with $\omega=\omega(x)$. Now let $\widehat{f}_2=\triangle^{k-1}f_2$, so that $\triangle \widehat{f}_2=p_2$, and let  $\tilde{\omega}(x)$ be such that $\tilde{\omega}''(x)=\Delta \tilde{\omega}=\omega(x)$, $\tilde{\omega}'(0)=\tilde{\omega}(0)=0$. Since $\Delta \tilde{\omega}=\omega$, by Lemma \ref{lambdas} we have $\Delta(\tilde{\omega} \circ \phi)=\mu^2\cdot (\tilde{\omega} \circ \phi)$. Therefore, denoting $\tilde{p}_2=\frac{1}{\mu^2}(\tilde{\omega} \circ \phi)$, we have \[\Delta \tilde{p}_2=\frac{1}{\mu^2}\cdot \Delta (\tilde{\omega}\circ \phi)=\frac{1}{\mu^2}\cdot \mu^2\cdot (\Delta \tilde{\omega} \circ \phi)=\Delta\tilde{\omega}\circ \phi=\omega\circ \phi=p_2\]

Since by construction $\Delta \widehat{f}_2=p_2$, we get that
$h_2=\widehat{f}_2-\tilde{p}_2$ is a harmonic function.

\begin{remark} \label{rem-computation} From a computational point of view, we do not need to know the scaling constant $\mu$ in order to compute $\tilde{p}_2$ or $h_2$. Instead, we directly find $\tilde{p}_2$ imposing the following two conditions: (1) $\tilde{p}_2$ is a polynomial of degree equal to $\mbox{deg}(p_2)+2$; (2) $\tilde{p}_2$ is the composition of a univariate polynomial $\gamma(x)$ (of degree $\mbox{deg}(p_2)+2$) with $\phi_1(x,y)$ (which is known); (3) $\Delta \tilde{p}_2=p_2$. These conditions lead to a linear system of equations whose unknowns are the coefficients of $\gamma(x)$.
\end{remark}

Now let $p_1$ be as in Eq. \eqref{exp-p}, and $\tilde{p}_1$ as in Eq. \eqref{p-tilde}, and let $q=\ell-k$. Since by definition $p_1=\triangle^k f_1$ and $p_2=\triangle^k f_2$, from Lemma \eqref{lambdas} we have $p_1\circ \phi=\lambda_k\cdot p_2$. Then we have
\[
p_2(x,y)=\frac{1}{\lambda_k}\cdot \left[\phi_1(x,y)^{2q+1}+a_{2q}\phi_1(x,y)^{2q}+\cdots\right].
\]
Since $p_2=\omega(\phi_1(x,y))$, we observe that $\omega(x)=\frac{1}{\lambda_k}\cdot p_1(x)$. Therefore, $\tilde{\omega}(x)=\frac{1}{\lambda_k}\cdot \tilde{p}_1(x)$. Moreover, by definition $\tilde{p}_2=\frac{1}{\mu^2}\cdot (\tilde{\omega}\circ \phi)$, so $\tilde{p}_2=\frac{1}{\mu^2\lambda_k}\cdot (\tilde{p}_1\circ \phi)$. Since using Eq. \eqref{recursion} we get $\mu^2\lambda_k=\lambda_{k-1}$, we finally obtain $\tilde{p}_1\circ \phi=\lambda_{k-1}\cdot\tilde{p}_2$.

\begin{proposition} Let $h_i=\widehat{f}_i-\tilde{p}_i$, with $i=1,2$, where $\widehat{f}_i$ and $\tilde{p}_i$ are defined as before. Then $\simm(f_1,f_2)\subset \simm(\tilde{f}_1,\tilde{f}_2)$.
\end{proposition}

\begin{proof} Let $\phi$ satisfy $f_1\circ \phi=\lambda f_2$, for some nonzero constant $\lambda$, and let us see that $h_1\circ \phi=\lambda^{\star} \cdot h_2$ for some other constant $\lambda^{\star}$, possibly different from $\lambda$. In order to do this, since $\widehat{f}_1=\Delta^{k-1} f_1$ and $\widehat{f}_2=\Delta^{k-1}f_2$, from Eq. \eqref{laplacian-k} we deduce that $\widehat{f}_1\circ \phi=\lambda_{k-1} \cdot\widehat{f}_2$. Furthermore, we have seen that $\tilde{p}_1\circ \phi=\lambda_{k-1} \cdot \tilde{p}_2$. Thus,
\[
h_1\circ \phi=(\widehat{f}_1-\tilde{p}_1)\circ \phi=\widehat{f}_1\circ \phi-\tilde{p}_1\circ \phi=\lambda_{k-1}\cdot\widehat{f}_2-\lambda_{k-1}\cdot\tilde{p}_2=\lambda_{k-1} \cdot(\widehat{f}_2-\tilde{p}_2)=\lambda_{k-1} \cdot h_2
\]
Therefore $\lambda^{\star}=\lambda_{k-1}$, and the result follows.
\end{proof}

If the degrees of both $h_1,h_2$ are higher than 2, we are in the case discussed before. If the degrees of both $h_1,h_2$ are 1, we observe that $h_1$ does not depend only on $x$, and $h_2$ is not a function only of $ax+by+c$, i.e. $h_2$ is not the composition of a univariate polynomial with $ax+by+c$. Furthermore, the similarity $\phi$ transforms the curve defined by $g_1(x,y)=x\cdot h_1(x,y)$ into the curve defined by $g_1(x,y)=(ax+by+c)\cdot h_2(x,y)$, and therefore we are in the case of two conic curves (more precisely, two pairs of secant lines), which can be solved by elementary methods.  If the degrees of both $h_1,h_2$ have degree 2 and $h_1,h_2$ are not two circles, the problem can be solved by elementary methods. If $h_1,h_2$ are are two circles, computing the similarities transforming product of the line $x=0$ and the circle defined by $h_1(x,y)$ into the product of the line $ax+by+c=0$ and the circle defined by $h_2(x,y)$ is straightforward.

Summarizing, one can observe that in fact, for each polynomial $f_i$, with $i=1,2$, we are applying Algorithm 2 to replace $f_i$ by a harmonic polynomial $h_i$. When $h_i$ has degree 1 or 2, in turn we replace $h_i$ by either a pair of secant lines, or the product of a circle and a conic. The practical efficiency is very high, since deriving the $h_i$ is fast, and comparing the $h_i$ boils down to either applying Algorithm 5 (see Remark \ref{efficiency-5}), or applying elementary methods in the case when the $h_i$ are pairs of secant lines, or unions of a line and a circle.


\begin{example} Let
\[
f_1(x,y)=-\frac{1}{120}(y^2-1)^2+\frac{1}{120}x^2(x-1)(x-2)(x+5)-\frac{1}{120}x^4
\]
define the curve ${\mathcal C}_1$, called \emph{stirrup curve}(see \cite{W}). Applying the transformation
\[
\{x:=y,\mbox{ }y:=x,\}
\]
i.e. a rotation of $\frac{\pi}{2}$ radians about the origin, we get another curve ${\mathcal C}_2$, defined by the polynomial
\[
f_2(x,y)=-\frac{1}{120}(x^2-1)^2+\frac{1}{120}y^2(y-1)(y-2)(y+5)-\frac{1}{120}y^4
\]
In particular, calling $\phi(x,y)=(y,x)$, we have $f_1\circ \phi=f_2$. The curves ${\mathcal C}_1$ (in blue) and ${\mathcal C}_2$ (in red) are plotted in Figure \ref{fig:sim}.

\begin{figure}
\begin{center}
\includegraphics[scale=0.3]{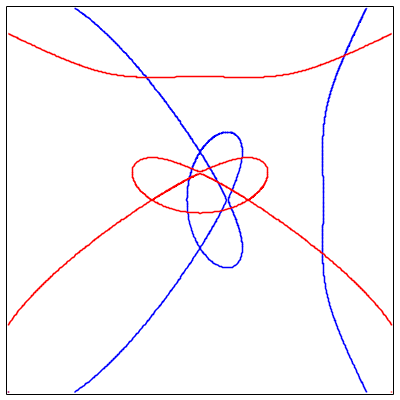}
\caption{Two similar curves.}\label{fig:sim}
\end{center}
\end{figure}

The laplacian chains starting with $f_1,f_2$ are
\[
\begin{array}{l}
  f_1\longmapsto\triangle f_1\longmapsto\triangle^2 f_1=x\\
	f_2\longmapsto\triangle f_2\longmapsto\triangle^2 f_2=y.
\end{array}
\]
Then we have $\widehat{f}_1=\triangle f_1=\frac{1}{6}x^3+\frac{1}{10}x^2-\frac{13}{20}x+\frac{1}{5}-\frac{1}{10}y^2$ and $p_1(x)=x$, so $\tilde{p}_1(x)=\frac{1}{3}x^6$. Therefore,
\[
\tilde{f}_1(x,y)=\widehat{f}_1(x,y)-\tilde{p}_1(x)=\frac{1}{10}x^2-\frac{13}{20}x+\frac{1}{5}-\frac{1}{10}y^2,
\]
which satisfies that $\Delta \tilde{f}_1=0$. One can easily see that $\tilde{f}_1(x,y)$ is a hyperbola centered at the point $\left(\frac{13}{4},0\right)$, where the major axis is parallel to the $x$-axis, and the minor axis is parallel to the $y$-axis. Analogously, $\widehat{f}_2(x,y)=-\frac{1}{10}x^2+\frac{1}{5}+\frac{1}{6}y^3+\frac{1}{10}y^2-\frac{13}{20}y$ and $\tilde{p}_2=\tilde{p}_2(y)=\frac{1}{6}y^3$, so
\[
\tilde{f}_2(x,y)=-\frac{1}{10}x^2+\frac{1}{5}+\frac{1}{10}y^2-\frac{13}{20}y,
\]
which is also harmonic, and defines a hyperbola centered at the point $\left(0,\frac{13}{4}\right)$, where the major axis is parallel to the $y$-axis, and the minor axis is parallel to the $x$-axis. One can readily see that calling $\phi(x,y)=(y,x)$, we have $\tilde{f}_1\circ \phi=\tilde{f}_2$. Notice that the hyperbolas defined by $\tilde{f}_1$ and $\tilde{f}_2$ are also related by the similarity $\phi_1(x,y)=(y,-x)$, which also maps ${\mathcal C}_1$ onto ${\mathcal C}_2$.
\end{example}

\vspace{0.3 cm}
\noindent \fbox{$\deg\triangle^\ell f_1, \triangle^\ell f_2=2$.} in this case, $g_1=\triangle^{\ell}f_1(x,y)=0$ and $g_2=\triangle^{\ell}f_2(x,y)=0$ are conic sections. If they are pairs of parallel lines we can reduce the problem to the previous case, and if they are not circles we can check whether or not they are similar by elementary methods. If we have two circles defined by $g_1,g_2$, centered at points $P_1,P_2$ then we proceed as in the case $\deg\triangle^\ell f=2$ addressed in Subsection \ref{red-symmetries} for the computation of symmetries. More precisely, for each $g_i$ we: (1) apply a translation $T_{{\bf v}_i}$ so that $P_i$ is moved to the origin; (2) $h_i=\widehat{f}_i-\tilde{\omega}_i$ where $\widehat{f}_i$ is the last element in the laplacian chain starting with $f_i$ that is not a product of concentric circles centered at $P_i$; (3) $\tilde{\omega}_i=\tilde{p}_i\circ T_{-{\bf v}_i}$, where $\tilde{p}_i$ corresponds to the polynomial $\tilde{p}$ we constructed in the symmetries case, appearing in Proposition \ref{prop-r}. Reasoning in an analogous way to the case before, one can check that the polynomials $h_i$ are harmonic, and that $\simm(f_1,f_2)\subset \simm(\tilde{f}_1,\tilde{f}_2)$. Summarizing, either $g_1,g_2$ are nondegenerate conics, in which case comparing whether or not they are similar can be done by elementary Linear Algebra methods, or we find two harmonic polynomials $h_1,h_2$ that we compare with Algorithm 5.

The whole procedure to compute the similarities is given in Algorithm 6.

\begin{algorithm}[t!]
\begin{algorithmic}[1]
\REQUIRE Two square-free polynomials $f_1(x,y),f_2(x,y)$ of the same degree $n$.
\ENSURE The similarities between the curves $\av{C}_1,\av{C}_2$ defined by $f_1,f_2$.
\FOR{i=1,2}
\FOR{$s=1$ to $\lceil \frac{n}{2} \rceil $}
\STATE{compute $\Delta^s f_i$}
\IF{$\deg(\Delta^s f_i)=0$}
\STATE{$\ell_i:=s-1$; {\bf break}}
\ENDIF
\ENDFOR
\ENDFOR
\IF{$\ell_1\neq \ell_2$}
\STATE{{\bf return} {\tt No similarities found}}
\ENDIF
\IF{$\deg(\Delta^\ell f_i)>2$}
\STATE{$h_i=\Delta^{\ell_i} f_i$, $i=1,2$}
\ENDIF
\IF{$\deg(\Delta^\ell f_i)=1$}
\STATE{find the polynomials $h_1,h_2$ using Algorithm 2}
\ENDIF
\IF{$\deg(\Delta^\ell f_i)=2$}
\STATE{find the polynomials $h_1,h_2$ using Algorithm 3}
\ENDIF
\IF{$\deg(h_i)=1$, $i=1,2$ or $\deg(h_i)=2$, $i=1,2$ and the $h_i$s are circles}
\STATE{replace $h_i$ by a pair of secant lines, if $\deg(h_i)=1$ for $i=1,2$, and by the product of $h_i$ times a secant line, if $h_i$, for $i=1,2$, is a circle.}
\ENDIF
\STATE{compute the similarities between $h_1,h_2$ using either Algorithm 5, or elementary methods.}
\FOR{each similarity$\phi$ found in the previous step}
\STATE{check whether $f_1\circ \phi=\lambda f_2$}
\ENDFOR
\STATE{{\bf return} the similarities of $f$, or the message {\tt No similarities found}.}
\end{algorithmic}
\caption{{\tt Similarities General}}\label{sim-case-general}
\end{algorithm}

\begin{remark}\label{rem-efficiency-again}
In order to give an idea of the efficiency of the method, we picked, in the computer algebra system {\tt Maple} 18, a random, dense, polynomial $f_1(x,y)$ of degree 30, with coefficients of order up to $2^6$, and applied the similarity $\phi(x,y)=(x-2y,y=2x+y)$ to obtain another polynomial $f_2(x,y)$, also dense, of degree 30 and coefficients of order up to $2^{50}$. The double laplacian chain yields two ellipses, and we test whether or not these ellipses are similar using the Maple package {\tt geometry}, getting an affirmative response. The whole process takes 0.702 seconds in the personal laptop mentioned in Remark \ref{efficiency-5}.
\end{remark}

\section{Conclusion and Further Research}\label{sec-conclusion}

In this paper we have presented a novel, efficient method to compute the symmetries of a planar algebraic curve, and the similarities, if any, between two algebraic curves. The method is based on the reduction of the problem to the analogous problem on harmonic polynomials, taking advantage of the fact that the laplacian operator commutes with orthogonal transformations. It is natural to wonder if the method can be extended to algebraic surfaces. Certainly, taking laplacians we can reduce the problem of computing the symmetries of an algebraic surface (similarities are analogous) to the problem of computing either the symmetries of an algebraic surface defined by a harmonic polynomial in the variables $x,y,z$, or the symmetries of a quadric, or the symmetries of a plane. If we reach a quadric which is not a surface of revolution, then we can certainly solve the problem. But in the other cases, the strategy is not clear, and deserves further research. These are questions that we would like to investigate in the future.

\vspace{0.2 cm}
\noindent{\bf Acknowledgements.} Juan G. Alc\'azar is supported by the Spanish Ministerio de Econom\'{\i}a y Competitividad and by the European Regional Development Fund (ERDF), under the project  MTM2014-54141-P, and is a member of the Research Group {\sc asynacs} (Ref. {\sc ccee2011/r34}). M. L\'{a}vi\v{c}ka and J. Vr\v{s}ek are  supported by the project LO1506 of the Czech Ministry of Education, Youth and Sports.
Additionally, in order to develop the results of the paper, Juan~G.~Alc\'azar was also partially supported by the project LO1506 during his stay in Plze\v{n}, and Jan Vr\v{s}ek was also partially supported by a mobility grant from the Universidad de Alcal\'a.

\end{document}